\documentclass[11pt]{article}
\usepackage{amssymb, amsfonts, amsmath, amsthm}
\usepackage{color}
\usepackage{verbatim}
\usepackage{authblk}
\usepackage{cite}
\usepackage{longtable}

\newtheorem{theorem}{Theorem}[section]

\newtheorem{conjecture}[theorem]{Conjecture}

\begin{document}
	
	\title{\textbf{Finding congruences with the WZ method}\footnote{Supported by the National Key Research and Development Program of China (2023YFA1009401)}}
	\author{Li-Quan Feng and Qing-Hu Hou}
	\affil{School of Mathematics and KL-AAGDM \break Tianjin University \break Tianjin 300350, China \break\texttt{feng\_liquan@tju.edu.cn, qh\_hou@tju.edu.cn}}	
	\maketitle
	
	\begin{abstract}
		We utilize the Wilf-Zeilberger (WZ) method to establish congruences related to truncated Ramanujan-type series. By constructing hypergeometric terms \(f(k, a, b, \ldots)\) with Gosper-summable differences and selecting appropriate parameters, we derive several congruences modulo \(p\) and \(p^2\) for primes \(p > 2\). For instance, we prove that for any prime \(p > 2\),
		\[
		\sum_{n=0}^{p-1} \frac{10n+3}{2^{3n}}\binom{3n}{n}\binom{2n}{n}^2 \equiv 0 \pmod{p},
		\]
		and
		\[
		\sum_{n=0}^{p-1} \frac{(-1)^n(20n^2+8n+1)}{2^{12n}}\binom{2n}{n}^5 \equiv 0 \pmod{p^2}.
		\]
		These results partially confirm conjectures by Sun and provide some novel congruences.

		
	\end{abstract}
	

\section{Introduction}
In 1914, Ramanujan \cite{Ram14} discovered $17$ infinite series representations of $1/\pi$, which take the form 
\[
\sum_{n=0}^\infty (an+b)  \frac{t_n}{m^n} = \frac{c}{\pi},
\]
where $a,\ b,\ c,\ m$ are constants and $t_n$ is a product of the binomial coefficients ${2n \choose n}, {3n \choose n}, {4n \choose 2n}$.  In 1997, Van Hamme \cite{Van97} investigated the 
partial sums  of Ramanujan-type series for $1/\pi$, and proposed $13$ conjectured  supercongruences, labeled as (A.2)--(M.2). These supercongruences can be viewed as $p$-adic analogues of their respective infinite series identities. So far, all of these $13$ supercongruences have been verified using various methods \cite{Mor08,Zu09,Long11,Kil06,Long16,Mc08,Osb16,Swi15}.
Among these methods, WZ method not only can provide extremely succinct proofs of known hypergeometric identities, but also can discover new hypergeometric identities. It was posed by H.S. Wilf, and D. Zeilberger \cite{WZ90a} in 1990.

Gessel \cite{Ges95} introduced a systematic approach to construct WZ pairs and established a lot of terminating hypergeometric identities. Based on this construction, Au \cite{Au25} proved many nonterminating hypergeometric identities. In the present paper, we aim to establish congruences utilizing this framework.

Recall that a hypergeometric term $t_k$ is called Gosper summable with respect to $k$ if there exists a hypergeometric term $z_k$ such that 
\[
t_k = \Delta_k z_k = z_{k+1} - z_k.
\]
In a manner similar to the works of Gessel and Au, we derive WZ pairs from classical hypergeometric identities. More specifically, we construct hypergeometric terms $f(k, a, b, \ldots)$ with parameters $a,b,\ldots$ such that 
\[
\Delta_a f(k,a,b,\ldots),\  \Delta_b f(k, a,b,\ldots),\ldots
\]
are all Gosper summable with respect to $k$. For any integers $u_a,\ u_b,\ldots$ and any real numbers $v_a,\ v_b,\ldots$, we define 
\[
F(n,k) =  f(k, u_a n+v_a, u_b n + v_b, \ldots).
\] 
It is shown \cite[Theorem 3.3]{Au25} that there exists a hypergeometric term $G(n,k)$ such that $(F(n,k), G(n,k))$ forms a WZ pair, i.e.,
\[
\Delta_n F(n,k)  = \Delta_k G(n,k).
\]
We will  identify parameters $u_a,\ u_b,\ldots$ and $v_a,\ v_b,\ldots$ such that
\[
F(0,k) \equiv F(p,k) \pmod{p^r}, \quad \forall\, 0 \le k < p.
\]
This leads to the result
\[
0 \equiv \sum_{n=0}^{p-1} F(n,k) = \Delta_k \left( \sum_{n=0}^{p-1} G(n,k) \right) \pmod{p^r}, \quad \forall\, 0 \le k < p.
\]
Finally, by selecting an appropriate $k_0$, for instance $k_0=(p-1)/2$, we can evaluate $\sum_{n=0}^{p-1} G(n,k_0) \pmod{p^r}$.
Consequently, we obtain a congruence
\[
\sum_{n=0}^{p-1} G(n,0) \equiv \sum_{n=0}^{p-1} G(n,k_0) \pmod{p^r}.
\]
Using this methodology, we establish several congruences associated with truncated Ramanujan-type series. 

Some of these congruences have been proven to exist in more general forms. For instance, we derive
\begin{equation}
	\sum_{n=0}^{p-1} \frac{6n^2+6n+1}{2^{4n}(n+1)^2 }\binom{2n}{n}^3  \equiv\sum_{n=0}^{p-1}\frac{3n+1}{2^{4n}}\binom{2n}{n}^3\equiv 0\pmod{p}, \label{3n+1}	
\end{equation}
which were proved by Hu in \cite{Hu17}.  More general, Sun \cite[Conjecture 5.1]{Sun11} conjectured that for any $r \in \mathbb{Z}^{+}$ and $p>3$,
\begin{equation} \label{3n+1r}
	\sum_{n=0}^{p^r-1}\frac{3n+1}{2^{4n}}\binom{2n}{n}^3\equiv p^r+\frac{7}{6}p^{r+3} B_{p-3}\pmod{p^{r+4}}.
\end{equation}
where $B_0,B_1,B_2,...$ are Bernoulli numbers defined by 
\[B_0=1,\quad\sum_{k=0}^{n-1}\binom{n}{k} B_k=0 \quad\text{when}\quad n>1.\]
Guillera and Zudilin \cite{Gui12} employed the WZ method to prove \eqref{3n+1r} modulo $p^3$. Guo \cite{Guo20} and Guo and Schlosser \cite{Guo201} derived two $q$-analogues of \eqref{3n+1r} which led to the confirmation of Sun's conjecture modulo $p^{r+2}$. Subsequently, Wang and Hu \cite{Wang23} extended these results to $p^{r+4}$, thereby providing a complete proof of Sun's conjecture.

Some of these congruences have been conjectured to exist in more general forms and we thus partially prove these conjectures.

\begin{theorem}\label{binom{3n}{n}binom{2n}{n}^2}
	For any prime $p>2$, we have
	\begin{equation}
		\sum_{n=0}^{p-1}\frac{10n+3}{2^{3n}}\binom{3n}{n}\binom{2n}{n}^2\equiv 0\pmod{p}.\label{10n+3}
	\end{equation}
	\begin{equation}
		\sum_{n=0}^{p-1}\frac{11n+3}{2^{6n}}\binom{3n}{n}\binom{2n}{n}^2\equiv 0\pmod{p}.\label{11n+3}
	\end{equation}			
\end{theorem}
We note that Sun \cite[Conjecture 5.5, 5.4]{Sun11} proposed general conjectures for the sums in \eqref{10n+3} and \eqref{11n+3}, where the summation range extends up to ${p^r - 1}$ and is considered modulo $p^{r+4}$.

\begin{theorem}\label{binom{3n}{n}binom{2n}{n}^4}
	For any prime $p>2$, we have
	\begin{equation}
		\sum_{n=0}^{p-1}\frac{(-1)^n(28n^2+18n+3)}{2^{6n}}\binom{3n}{n}\binom{2n}{n}^4\equiv 0\pmod{p^2}.\label{28n^2+18n+3}
	\end{equation}
	\begin{equation}
		\sum_{n=0}^{p-1} \frac{74n^2+27n+3}{2^{12n}}\binom{3n}{n}\binom{2n}{n}^4\equiv 0\pmod{p}.\label{74n^2+27n+3}
	\end{equation}	
\end{theorem}
We remark that Sun \cite[Conjecture 1.4]{Sun10} provided a conjecture on the sum in \eqref{28n^2+18n+3} modulo $p^6$ and
conjectured \eqref{74n^2+27n+3} modulo $p^6$ in \cite[Conjecture 31]{Sun19}.

\begin{theorem}\label{42n^2+23n+3}
	For any prime $p>2$, we have			
	\begin{multline}
		\sum_{n=0}^{p-1} \frac{42n^2+23n+3}{2^{4n}(2n+1)}\binom{4n}{2n}^2\binom{2n}{n} \equiv\\ \sum_{n=0}^{p-1}
		\frac{252n^3 + 384n^2 + 139n + 15}{2^{4n}} \binom{4n}{2n}^2 \binom{2n}{n} \equiv 0 \pmod{p}.\label{252n^3 + 384n^2 + 139n + 15}
	\end{multline}
\end{theorem}
We remark that Sun \cite[Conjecture 4.14]{Sun23} conjectured the first congruence modulo $p^5$ in 2023.

\begin{theorem}\label{binom{4n}{2n}binom{3n}{n}binom{2n}{n}}
    For any prime $p>2$, we have			
	\begin{equation}
		\sum_{n=0}^{p-1} \frac{92n^2+61n+9}{2^{6n}(2n+1)}\binom{4n}{2n}\binom{3n}{n}\binom{2n}{n}  \equiv 0 \pmod{p}.\label{92n^2+61n+9}
	\end{equation}
\end{theorem}
We remark that Sun \cite[Conjecture 4.15]{Sun23} conjectured this congruence modulo $p^3$.

\begin{theorem}\label{binom{4n}{2n}binom{3n}{n}binom{2n}{n}^3}
	For any prime $p>2$, we have
	\begin{equation}
		\sum_{n=0}^{p-1}\frac{(-1)^n(172n^2+75n+9)}{2^{12n}}\binom{4n}{2n}\binom{3n}{n}\binom{2n}{n}^3\equiv 0\pmod{p}.\label{172n^2+75n+9}
	\end{equation}
\end{theorem}
We remark that Guillera and Zudilin \cite[(9)]{Gui12} conjectured this congruence modulo $p^5$.

\begin{theorem}\label{binom{6n}{3n}binom{3n}{n}binom{2n}{n}}
	For any prime $p>2$, we have
	\begin{equation}
		\sum_{n=0}^{p-1}\frac{(-1)^n(154n+15)}{2^{15n}}\binom{6n}{3n}\binom{3n}{n}\binom{2n}{n}\equiv 0\pmod{p}.\label{154n+15}
	\end{equation}
\end{theorem}		
We remark that Sun \cite[Conjecture 5.8]{Sun11} proposed general conjecture that the summation range extends up to ${p^r - 1}$ and is considered modulo $p^{r+3}$.

We also discover some congruences appear to be new, as given in Theorems~\ref{binom{2n}{n}^5}--\ref{binom{4n}{2n}^2binom{2n}{n}}.
\begin{theorem}\label{binom{2n}{n}^5}
	For any prime $p>2$, we have
	\begin{equation}
		\sum_{n=0}^{p-1}\frac{(-1)^n(20n^2+8n+1)}{2^{12n}}\binom{2n}{n}^5\equiv 0\pmod{p^2}.\label{20n^2+8n+1}
	\end{equation}
	\begin{equation}
		\sum_{n=0}^{p-1}\frac{(-1)^n(820n^2+180n+13)}{2^{20n}}\binom{2n}{n}^5\equiv 0\pmod{p}.\label{820n^2+180n+13}
	\end{equation}
\end{theorem}		
We remark that Guillera \cite[Identity 8]{Gui08} \cite{Gui02} obtained the series identities
\[\sum_{n=0}^{\infty}\frac{(-1)^n(20n^2+8n+1)}{2^{12n}}\binom{2n}{n}^5=\frac{8}{\pi^2},\]
\[\sum_{n=0}^{\infty}\frac{(-1)^n(820n^2+180n+13)}{2^{20n}}\binom{2n}{n}^5=\frac{128}{\pi^2}.\]
\begin{theorem}\label{binom{4n}{2n}^2binom{2n}{n}}
	For any prime $p>2$, we have
	\begin{multline}
		\sum_{n=0}^{p-1} \frac{(-1)^n(40n^2+24n+3)}{2^{8n}(2n+1)}\binom{4n}{2n}^2\binom{2n}{n} \equiv\\ \sum_{n=0}^{p-1}
		\frac{(-1)^n(80n^3 + 96n^2 + 30n + 3)}{2^{8n}} \binom{4n}{2n}^2 \binom{2n}{n} \equiv 0 \pmod{p}.\label{80n^3 + 96n^2 + 30n + 3}
	\end{multline} 
	\begin{multline}
		\sum_{n=0}^{p-1} \frac{48n^2+32n+3}{2^{12n}(2n+1)}\binom{4n}{2n}^2\binom{2n}{n} \equiv\\ \sum_{n=0}^{p-1}
		\frac{-96n^3 + 48n^2 + 26n + 3}{2^{12n}} \binom{4n}{2n}^2 \binom{2n}{n} \equiv 0 \pmod{p}.\label{-96n^3 + 48n^2 + 26n + 3}
	\end{multline}
	
\end{theorem}

\section{Proofs of congruences}
In this section, we show the general frame for discovering and proving the congruences by WZ method. 

We illustrate the process by the derivation of Theorem~\ref{binom{3n}{n}binom{2n}{n}^2}.
\begin{proof}[Proof of Theorem \ref{binom{3n}{n}binom{2n}{n}^2}]
	We start with Bailey's theorem \cite{Bai35},
	\begin{equation}
		_2F_1\left[\begin{matrix}\;a,\ 1-a \\ b \end{matrix}\;;\ \frac{1}{2}\;\right]
		=\Gamma{\begin{bmatrix}\frac{b}{2},\ \frac{b+1}{2}\\
				\frac{a+b}{2},\ \frac{1-a+b}{2}\end{bmatrix}}.\label{Bai2F1}
	\end{equation}
	Here we use the standard notation for hypergeometric series
	\[_{r}F_{r-1}\left[\begin{matrix}a_1,\ a_2,\ \dots,\ a_r \\b_1,\ b_2,\ \dots,\ b_{r-1} \end{matrix}\;;\ z\;\right]=\sum_{k=0}^{\infty}\frac{(a_1)_k(a_2)_k\dots(a_r)_k}{(b_1)_k(b_2)_k\dots(b_{r-1})_k}\frac{z^k}{k!},\]
	where $(a)_k = a (a+1) \cdots (a+k-1)$ is the rising factorial. And we briefly denote
	\[\Gamma\left[\begin{matrix}a,\ b,\ \dots,\ c \\a',\ b',\ \dots,\ c' \end{matrix}\right]=\frac{\Gamma(a)\Gamma(b)\dots\Gamma(c)}{\Gamma(a')\Gamma(b')\dots\Gamma(c')}.\]
	
	Dividing the right hand side of \eqref{Bai2F1}, we find that 
	\begin{equation}
		F_0(a,b,k) = 2^{2a-k}\cdot \Gamma
		\left[ 
		\begin{matrix}
			2a+k,\ k+1-2a,\ a+b,\ 2b-2a+1 \\
			2b+k,\ k+1,\ b-a+1
		\end{matrix}
		\right]
		\label{2f1}
	\end{equation}
	satisfies the condition that
	\[\Delta_a F_0(a,b,k),\ \Delta_b F_0(a,b,k)\]
	are all Gosper summable. Therefore, for any integers $u_a,\ u_b,\ u_k$ and real numbers $v_a,\ v_b,\ v_k$,
	$\Delta_n F(n,k)$ is Gosper summable, where
	\begin{equation}\label{F}
		F(n,k) = F_0(u_an+v_a,\ u_bn+v_b,\  u_kn+k+v_k).
	\end{equation}
	That is, there exists $G(n,k)$ such that
	\begin{equation}\label{WZ}
		\Delta_n F(n,k) = \Delta_k G(n,k).
	\end{equation}
	
	Now we try all parameters in the following ranges
	\[
	u_a,\ u_b,\ u_k \in \{-1, 0, 1\} \quad \mbox{and} \quad   v_a,\ v_b,\ v_k \in \{0, 1/2, -1/2\}
	\]
	and choose those parameters such that 
	\[
	F(0,k) \equiv F(p,k) \pmod{p^r}, \quad \forall\, 0 \le k < p.
	\]
	To derive \eqref{10n+3} of Theorem~\ref{binom{3n}{n}binom{2n}{n}^2}, we take $u_a=u_k=1,\ u_b=0$, and $v_a=v_k=0,\ v_b=1/2$. After removing constants and simplifying the product of Gamma functions by the relations
	\[
	\Gamma(x) \Gamma(1-x) = \frac{\pi}{\sin \pi x}  \quad \mbox{and} \quad \Gamma (x+1/2) = \sqrt{\pi} \cdot 2^{1-2x} \frac{\Gamma (2x)}{\Gamma(x)},
	\]
	we obtain
	\[F(n,k)=\frac{(-1)^{k}n(3n+k-1)!(2n)!}{2^{3n+k}(n-k-1)!(n+k)!^2n!^2}.\]
	By Gosper's algorithm, we find that
	\[G(n,k)=\frac{(-1)^{k+1}(10n^2+6nk+3n+k)(3n+k-1)!(2n)!}{2^{3n+k+1}(n-k)!(n+k)!^2n!^2}.\]
	
	Clearly, $F(0,k) = 0$. 
	We further consider the $p$-adic evaluation $v_p(F(p,k))$ of $F(p,k)$ for $0 \leq k \leq p-1$. We assume that all factorials appear in the expression of $F(p,k))$ are strictly less than $p^2$ so that
	\[
	v_p(x!) =  \lfloor x/p \rfloor.
	\]
	Moreover, we assume that $p>2$ so that $v_p(2) = 0$. Notice that for $k=0$, $v_p(F(p,0))=1+2+2-(2+2)=1$; for $0<k\leq p-1$, $v_p(F(p,0))=1+3+2-(2+2)=2$. In summary, $v_p(F(p,k))\geq 1$ for $0\leq k\leq p-1$. Thus we have
	\[F(p,k)\equiv F(0,k)\equiv 0\pmod{p}\quad \text{for}\quad 0\leq k\leq p-1.\]
	
	Summing up the WZ equation \eqref{WZ} for $n$ from 0 to $p-1$, we derive that
	\[\sum_{n=0}^{p-1} G(n,k+1)-\sum_{n=0}^{p-1} G(n,k) = F(p,k) - F(0,k) \equiv 0 \pmod{p},\]
	implying that the congruence $\sum_{n=0}^{p-1} G(n,k) \pmod{p}$ is independent of $k$. 
	Now we evaluate the congruence for $k=\frac{p-1}{2}$. We have 
	\begin{eqnarray*}
		G(n,\frac{p-1}{2})&=&\frac{(-1)^{\frac{p+1}{2}}(10n^2+3np+\frac{p-1}{2})(3n+\frac{p-1}{2}-1)!(2n)!}{2^{3n+\frac{p+1}{2}}(n-\frac{p-1}{2})!(n+\frac{p-1}{2})!^2n!^2}\\
		&=&\frac{(-1)^{\frac{p+1}{2}}(10n^2+3np+\frac{p-1}{2})(3n+\frac{p-3}{2})!(2n)!(\frac{p-1}{2})!\binom{n}{\frac{p-1}{2}}}{2^{3n+\frac{p+1}{2}}(n+\frac{p-1}{2})!^2n!^3}.
	\end{eqnarray*}
	Notice that for $0\leq n<\frac{p-1}{2}$, $G(n,\frac{p-1}{2})=0$, since $\binom{n}{\frac{p-1}{2}}=0$. For $\frac{p-1}{2}\leq n\leq p-1$, the cases of $p=3,\ 5,\ 7$ can be verified directly. For $p>7$, the classification of $n$ according to the computation of $v_p(G(n,\frac{p-1}{2}))$  are summarized in the following table:
	\begin{center}
		\begin{tabular}{|c|c|}
			\hline
			The range of $n$ & $p$-adic evaluation $v_p(G(n,\frac{p-1}{2}))$ \\ \hline
			$n=\frac{p-1}{2}$ & $1$ \\ \hline
			$\frac{p-1}{2}<n<\frac{5p+3}{6}$ & $2+1-2=1$ \\ \hline
			$\frac{5p+3}{6}\leq n\leq p-1$ & $3+1-2=2$ \\ \hline
		\end{tabular}
	\end{center}
	In summary, we have $v_p(G(n,\frac{p-1}{2}))\geq 1$ for $0\leq n\leq p-1$. Thus, 
	\[G(n,\frac{p-1}{2})\equiv 0\pmod{p}\quad \text{for}\quad 0\leq n\leq p-1,\] 
	which implies that
	\[\sum_{n=0}^{p-1}G(n,\frac{p-1}{2})\equiv 0\pmod{p},\]
	therefore,
	\[\sum_{n=0}^{p-1} G(n,0) \equiv \sum_{n=0}^{p-1}G(n,\frac{p-1}{2})\equiv 0 \pmod{p}.\]
	Consequently, we obtain
	\begin{eqnarray*}
		\sum_{n=0}^{p-1}G(n,0)&=&\sum_{n=0}^{p-1}-\frac{10n+3}{2^{3n+1}\cdot3}\binom{3n}{n}\binom{2n}{n}^2\\
		&\equiv&\sum_{n=0}^{p-1}\frac{10n+3}{2^{3n}}\binom{3n}{n}\binom{2n}{n}^2\equiv 0 \pmod{p},
	\end{eqnarray*}
	thereby proving \eqref{10n+3} of Theorem~\ref{binom{3n}{n}binom{2n}{n}^2}.
	
	For \eqref{11n+3} of  Theorem~\ref{binom{3n}{n}binom{2n}{n}^2}, we consider the following $_7F_6$-series identities introduced by Chu \cite{Chu94}:
	\[\begin{gathered}
		\left._7F_6\left[
		\begin{array}{c}
			a-\frac{1}{2},\ \frac{2a+2}{3},\ 2b-1,\ 2c-1,\ 2+2a-2b-2c,\ a+n,\ -n \\
			\frac{2a-1}{3},\ 1+a-b,\ 1+a-c,\ b+c-\frac{1}{2},\ 2a+2n,\ -2n
		\end{array}\right.;\ 1\ \right] \\
		=\frac{(\frac{1}{2}+a)_n(b)_n(c)_n(a-b-c+\frac{3}{2})_n}{(\frac{1}{2})_n(1+a-b)_n(1+a-c)_n(b+c-\frac{1}{2})_n}.
	\end{gathered}\]
	Let $n\to\infty$, we find that
	\begin{multline}
		F_0(a,b,c,k) =\frac{(-1)^{a+b+c}(2a + 3k - 1)(2a + 1 - 2b - 2c + k)}{2^{4a-4b-4c+2k}(a - b - c)}\ 
		\\
		\Gamma	\left[ 
		\begin{matrix}
			b + c + k,\ b,\ c,\ 2a + 2k - 1,\ 2b + k - 1,\\
			a + k,\ 2b - 1,\ 2c - 1,\ a - b + k + 1,\\ 
		\end{matrix}
		\right. \\
		\left.
		\begin{matrix}
			2c + k - 1,\ 2a - 2b - 2c + k+1,\ b + c-a +1 \\
			a - c + k+1,\ 2b + 2c + 2k-1,\ k+1 
		\end{matrix}
		\right].
		\label{7f6}
	\end{multline}
	Similar to the proof of \eqref{10n+3}, we set $u_a=u_b=0,\ u_c=u_k=1$ and $v_a=v_b=v_k=0,\ v_c=1/2$ in \eqref{7f6} to get a WZ pair:
	\[F(n,k)=\frac{(-1)^{k}n(3n+3k-1)(2n+2k-2)!(3n+k-1)!(2n)!k!}{2^{6n+2k}(n+k-1)(n+k)!^2(n-k-1)!(2n+k-1)!n!^2(2k)!},\]
	\[G(n,k)=\frac{(-1)^{k+1}\alpha(n,k)(2n+2k-3)!(3n+k-1)!(2n)!(k-1)!}{2^{6n+2k+2}(n+k)!^2(n-k)!(2n+k)!n!^2(2k-2)!},\]
	where \[\alpha(n,k)=22n^3+(32k-3)n^2+(10k^2+2k-3)n+k^2-k.\]
	
	Clearly, $F(0,k) = 0$. For $n=p$, notice that
	\[\frac{1}{n}F(n,k)=\frac{(-1)^{k}(3n+3k-1)\binom{3n+k-1}{n}\binom{2n+2k-2}{n-k-1}\binom{n+3k-1}{n+k}\binom{2n}{n}}{2^{6n+2k}k(n+k-1)\binom{n+k}{k}}\]
	is a $p$-adic integer for $p\mid n$ and $1<k\leq p-1$; for $k=0$, \[v_p(F(p,0))=1+1+2+2-(2+1+2)=1;\] for $k=1$, \[v_p(F(p,1))=1+2+3+2-(1+2+2+2)=1.\] Therefore we have
	\[F(p,k)\equiv F(0,k)\equiv 0\pmod{p}\quad \text{for}\quad 0\leq k\leq p-1.\]
	
	Now we evaluate the congruence $\sum_{n=0}^{p-1} G(n,k) \pmod{p}$ for $k=\frac{p-1}{2}$. We have 
	\begin{eqnarray*}
		G(n,\frac{p-1}{2})&=&\frac{(-1)^{\frac{p-1}{2}}\alpha(n,\frac{p-1}{2})(2n+p-4)!(3n+\frac{p-3}{2})!(2n)!(\frac{p-3}{2})!}{2^{6n+p+1}(n+\frac{p-1}{2})!^2(n-\frac{p-1}{2})!(2n+\frac{p-1}{2})!n!^2(p-3)!} \\ 
		&=&\frac{(-1)^{\frac{p+1}{2}}\alpha(n,\frac{p-1}{2})(2n+p-4)!(3n+\frac{p-3}{2})!(2n)!(\frac{p-1}{2})\binom{n}{\frac{p-1}{2}}}{2^{6n+p+1}(n+\frac{p-1}{2})!^2(2n+\frac{p-1}{2})!n!^3\binom{p-3}{\frac{p-3}{2}}}.
	\end{eqnarray*}
	Notice that for $0\leq n<\frac{p-1}{2}$, $G(n,\frac{p-1}{2})=0$, since $\binom{n}{\frac{p-1}{2}}=0$. For $\frac{p-1}{2}\leq n\leq p-1$, the cases of $p=3,\ 5,\ 7$ can be verified directly. For $p>7$, the classification of $n$ according to the computation of $v_p(G(n,\frac{p-1}{2}))$ are summarized in the following table:
	\begin{center}
		\begin{tabular}{|c|c|}
			\hline
			The range of $n$ & $p$-adic evaluation $v_p(G(n,\frac{p-1}{2}))$ \\ \hline
			$n=\frac{p-1}{2}$ & $1+1-1=1$ \\ \hline
			$\frac{p+1}{2}\leq n<\frac{p+4}{2}$ & $1+2+1-(2+1)=1$ \\ \hline
			$\frac{p+4}{2}< n< \frac{3p+1}{4}$ & $2+2+1-(2+1)=2$ \\ \hline
			$\frac{3p+1}{4}\leq n< \frac{5p+3}{6}$ & $2+2+1-(2+2)=1$ \\ \hline
			$\frac{5p+3}{6}\leq n\leq p-1$ & $2+3+1-(2+2)=2$ \\ \hline
		\end{tabular}
	\end{center}
	In summary, $v_p(G(n,\frac{p-1}{2}))\geq 1$ for $0\leq n\leq p-1$. Thus
	\[G(n,\frac{p-1}{2})\equiv 0\pmod{p}\quad \text{for}\quad 0\leq n\leq p-1,\] which implies that
	\[\sum_{n=0}^{p-1}G(n,\frac{p-1}{2})\equiv 0\pmod{p}.\]
	Consequently, we obtain
	\[\sum_{n=0}^{p-1}G(n,0)=\sum_{n=0}^{p-1} \frac{22n^2-3n-3}{2^{6n+3}3n(n-1)(2n-1)}\binom{3n}{n}\binom{2n}{n}^2 \equiv 0 \pmod{p}.\]
	
	By the extended-Zeilberger's algorithm \cite{Chen09}, we find that the hypergeometric term
	\[
	t_n = \frac{11n+3}{2^{6n}}\binom{3n}{n}\binom{2n}{n}^2
	\]
	satisfies
	\[
	t_n - 48\cdot G(n,0) = \Delta_n \left(\frac{n^3(2n)!(3n)!}{2^{6n-5}(2n-1)(n-1)n!^5} \right).
	\]
	Summing over $n$ from $0$ to $p-1$, we derive that
	\[
	\sum_{n=0}^{p-1} \frac{11n+3}{2^{6n}}\binom{3n}{n}\binom{2n}{n}^2 \equiv \sum_{n=0}^{p-1} G(n,0) \equiv 0 \pmod{p},
	\]
	thereby proving \eqref{11n+3} of Theorem \ref{binom{3n}{n}binom{2n}{n}^2}.
\end{proof}
\begin{proof}[Proof of Theorem \ref{binom{3n}{n}binom{2n}{n}^4}]
	For \eqref{28n^2+18n+3} of Theorem \ref{binom{3n}{n}binom{2n}{n}^4}, we utilize Whipple's theorem \cite{Bai35},
	\[_3F_2\left[\begin{matrix}a,\ 1-a,\ b \\c,\ 1+2b-c \end{matrix}\;;\ 1\;\right]
	=\pi\cdot 2^{1-2b}\cdot \Gamma{\begin{bmatrix} c,\ 1+2b-c\\
			\frac{a+c}{2},\frac{1-a+c}{2},1+b-\frac{a+c}{2},1+b-\frac{1-a+c}{2}\end{bmatrix}}.\]
	We find that
	\begin{multline}
		F_0(a,b,c,k) = \Gamma
		\left[ 
		\begin{matrix}
			2a+k,\ k+1-2a,\ b+k,\ a+c,\  \\
			b,\ 2c+k,\ 2b-2c+k+1,\ k+1,\ c-a+1,\ 
		\end{matrix}
		\right. \\
		\left.
		\begin{matrix}
			2c-2a,\ c+1-a-b \\
			a-b+c,\ 2c-2a-2b+1
		\end{matrix}
		\right].
		\label{3f2}
	\end{multline}
	Setting $u_a=u_k=1,\ u_b=u_c=0$, and $v_a=v_k=0,\ v_b=v_c=1/2$ in \eqref{3f2} to get a WZ pair:
	\[F(n,k)=\frac{(-1)^{n+k}n^2(3n+k-1)!(2n+2k)!(2n)!^2}{2^{6n+2k}(n-k-1)!(n+k)!^4n!^4},\]
	\[G(n,k)=\frac{\alpha(n,k)(-1)^{n+k}(3n+k-1)!(2n+2k)!(2n)!^2}{2^{6n+2k+2}(n-k)!(n+k)!^4n!^4},\]
	where \[\alpha(n,k)=28n^3+28n^2k+18n^2+8nk^2+12nk+3n+2k^2+k.\]
	
	Clearly, $F(0,k) = 0$. For $n=p$, the classification of $k$ according to the computation of $v_p(F(p,k))$ are summarized in the following table:
	\begin{center}
		\begin{tabular}{|c|c|}
			\hline
			The range of $k$ & $p$-adic evaluation $v_p(F(p,k))$ \\ \hline
			$k=0$ & $2+2+2+4-(4+4)=2$ \\ \hline
			$0<k<\frac{p}{2}$ & $2+3+2+4-(4+4)=3$ \\ \hline
			$\frac{p}{2}<k\leq p-1$ & $2+3+3+4-(4+4)=4$ \\ \hline
		\end{tabular}
	\end{center}
	In summary, $v_p(F(p,k))\geq 2$ for $0\leq k\leq p-1$. Thus,
	\[F(p,k)\equiv F(0,k)\equiv 0\pmod{p^2}\quad \text{for}\quad 0\leq k\leq p-1.\] 
	Now we evaluate the congruence $\sum_{n=0}^{p-1} G(n,k) \pmod{p}$ for $k=\frac{p-1}{2}$. We have 
	\[G(n,\frac{p-1}{2})=\frac{\alpha(n,\frac{p-1}{2})(-1)^{n+\frac{p-1}{2}}(3n+\frac{p-1}{2}-1)!(2n+p-1)!(2n)!^2}{2^{6n+p+2}(n-\frac{p-1}{2})!(n+\frac{p-1}{2})!^4n!^4},\]
	where \[\alpha(n,\frac{p-1}{2})=28n^3+(14p+4)n^2+(2p^2+2p-1)n+\frac{p^2-p}{2}.\]
	Notice that for $0\leq n<\frac{p-1}{2}$, $G(n,\frac{p-1}{2})=0$. For $\frac{p-1}{2}\leq n\leq p-1$, the cases of $p =3,\ 5,\ 7$ can be verified directly. For $p>7$,  the classification of $n$ according to the computation of $v_p(G(n,\frac{p-1}{2}))$  are summarized in the following table:
	\begin{center}
		\begin{tabular}{|c|c|}
			\hline
			The range of $n$ & $p$-adic evaluation $v_p(G(n,\frac{p-1}{2}))$ \\ \hline
			$n=\frac{p-1}{2}$ & $1+1=2$ \\ \hline
			$\frac{p+1}{2}\leq n<\frac{5p+3}{6}$ & $2+2+2-4=2$ \\ \hline
			$\frac{5p+3}{6}\leq n\leq p-1$ & $3+2+2-4=3$ \\ \hline
		\end{tabular}
	\end{center}
	In summary, $v_p(G(n,\frac{p-1}{2}))\geq 2$ for $0\leq n\leq p-1$. Thus
	\[G(n,\frac{p-1}{2})\equiv 0\pmod{p^2}\quad \text{for}\quad 0\leq n\leq p-1,\] which implies that    	\[\sum_{n=0}^{p-1}G(n,\frac{p-1}{2})\equiv 0\pmod{p^2}.\]
	Consequently, we obtain
	\begin{eqnarray*}
		\sum_{n=0}^{p-1}G(n,0) & = & \sum_{n=0}^{p-1} \frac{(-1)^n(28n^2+18n+3)\binom{3n}{n}\binom{2n}{n}^4}{2^{6n+2}\cdot3}\\
		& \equiv & \sum_{n=0}^{p-1} \frac{(-1)^n(28n^2+18n+3)\binom{3n}{n}\binom{2n}{n}^4}{2^{6n}}\equiv 0 \pmod{p^2},
	\end{eqnarray*}
	thereby proving \eqref{28n^2+18n+3} of Theorem \ref{binom{3n}{n}binom{2n}{n}^4}.
	
	For \eqref{74n^2+27n+3} of Theorem \ref{binom{3n}{n}binom{2n}{n}^4}, we utilize Dougall's $_5F_4$ summation formula \cite{Dou07}
	\begin{multline*}
		_5F_4\left[\begin{matrix}1+\frac{a}{2},\  a,\  b,\  c,\  d\\ \frac{a}{2}, 1+a-b,\ 1+a-c,\ 1+a-d \end{matrix}\;;\ 1\;\right]
		\\
		=\Gamma{\begin{bmatrix} 1+a-b,\ 1+a-c,\ 1+a-d,\ 1+a-b-c-d\\1+a,\ 1+a-b-c,\ 1+a-b-d,\ 1+a-c-d\end{bmatrix}},
	\end{multline*}
	we find that 
	\begin{multline}
		F_0(a,b,c,d,k)  =(-1)^{b+c+d} (a+2k)\  \Gamma
		\left[ 
		\begin{matrix}
			-a+b+c+d,\\
			b,\ c,\ d,\ k+1,\ b+c-a,
		\end{matrix}
		\right. \\
		\left.
		\begin{matrix}
			a+k,\ b+k,\ c+k,\ d+k \\
			b+d-a,\ c+d-a,\ a+1-b+k,\ a+1-c+k,\ a+1-d+k\
		\end{matrix}
		\right].
		\label{5f4}
	\end{multline}
	Setting $u_a=u_k=0,\ u_b=u_c=u_d=1,$ and $v_a=v_b=v_c=v_d=0,\ v_k=1/2$ in \eqref{5f4} to get a WZ pair:
	\[F(n,k)=\frac{(-1)^k n^2 (2n+2k)!^3 (3n)!(2n-2k-2)!^3}{2^{12n}  (n+k)!^3 (n-1)!^3 (n-k-1)!^3 (2n)!^3},\]
	\[G(n,k)= \frac{(-1)^k\alpha(n,k) (2n+2k)!^3(3n)!(2n-2k)!^3}{2^{12n+14} (2n+1)!^3 (n+k)!^3n!^3(n-k)!^3},\]
	where
	\begin{multline*}
		\alpha(n,k)=592n^5+1104n^4-480n^3k^2+792n^3-576n^2k^2+272n^2\\
		+144nk^4-216nk^2+45n+48k^4-24k^2+3.
	\end{multline*}
	
	Clearly, $F(0,k) = 0$. For $n=p$, the classification of $k$ according to the computation of $v_p(F(p,k))$ are summarized in the following table:
	\begin{center}
		\begin{tabular}{|c|c|}
			\hline
			The range of $k$ & $p$-adic evaluation $v_p(F(p,k))$ \\ \hline
			$0\leq k< \frac{p}{2}-1$ & $2+6+3+3-(3+6)=5$ \\ \hline
			$\frac{p}{2}-1< k< \frac{p}{2}$ & $2+6+3-(3+6)=2$ \\ \hline
			$\frac{p}{2}< k\leq p-1$ & $2+9+3-(3+6)=5$ \\ \hline
		\end{tabular}
	\end{center}
	In summary, $v_p(F(p,k))\geq 2$ for $0\leq k\leq p-1$. Thus, 
	\[F(p,k)\equiv F(0,k) \pmod{p^2} \quad \text{for}\quad 0\leq k\leq p-1.\]
	
	Now we evaluate the congruence $\sum_{n=0}^{p-1} G(n,k) \pmod{p}$ for $k=\frac{p-1}{2}$. We have
	\[G(n,\frac{p-1}{2})=\frac{(-1)^{\frac{p-1}{2}}\alpha(n,\frac{p-1}{2})(2n+p-1)!^3(3n)!(2n-p+1)!^3}{2^{12n+14}(2n+1)!^3(n+\frac{p-1}{2})!^3n!^3(n-\frac{p-1}{2})!^3},\]
	where
	\begin{multline*} \alpha(n,\frac{p-1}{2})=592n^5+1104n^4+(-120p^2+240p+672)n^3\\+(-144p^2+288p+128)n^2 +(9p^4-36p^3+72p)n+3p^2(p-2)^2.
	\end{multline*}
	The classification of $n$ according to the computation of $v_p(G(n,\frac{p-1}{2}))$  are summarized in the following table:
	\begin{center}
		\begin{longtable}{|c|c|}
			\hline
			The range of $n$ & $p$-adic evaluation $v_p(G(n,\frac{p-1}{2}))$ \\ \hline
			$n=0$ & $2$\\ \hline
			$0< n< \frac{p}{3}$ & $3$\\ \hline
			$\frac{p}{3}\leq n< \frac{p-1}{2}$ & $3+1=4$ \\ \hline
			$n=\frac{p-1}{2}$ & $3+1-3=1$ \\ \hline 
			$\frac{p+1}{2}\leq n<\frac{2p}{3}$ & $6+1-(3+3)=1$ \\ \hline
			$\frac{2p}{3}\leq n\leq p-1$ & $6+2-(3+3)=2$ \\ \hline
		\end{longtable}
	\end{center}
	In summary, $v_p(G(n,\frac{p-1}{2}))\geq 1$ for $0\leq n\leq p-1$. Thus \[G(n,\frac{p-1}{2})\equiv 0\pmod{p}\quad \text{for}\quad 0\leq n\leq p-1,\] 
	which implies that \[\sum_{n=0}^{p-1}G(n,\frac{p-1}{2})\equiv 0\pmod{p}.\]
	Consequently, we obtain
	\begin{eqnarray*}
		\sum_{n=0}^{p-1}G(n,0)&=&\sum_{n=0}^{p-1} \frac{74n^2+27n+3}{2^{12n+4}}\binom{3n}{n}\binom{2n}{n}^4\\ &\equiv& \sum_{n=0}^{p-1} \frac{74n^2+27n+3}{2^{12n}}\binom{3n}{n}\binom{2n}{n}^4 \equiv 0 \pmod{p},
	\end{eqnarray*}
	thereby proving \eqref{74n^2+27n+3} of Theorem \ref{binom{3n}{n}binom{2n}{n}^4}.
\end{proof}   

\begin{proof}[Proof of Theorem \ref{42n^2+23n+3}]
	According to the Dixon–Kummer summation theorem \cite{Whi26},
	\[_4F_3\left[\begin{matrix}a,\ 1+\frac{a}{2},\ b,\ c \\\frac{a}{2},\ 1+a-b,\ 1+a-c \end{matrix}\;;-1\; \right]
	=\Gamma{\begin{bmatrix} 1+a-b,\ 1+a-c\\1+a,\ 1+a-b-c\end{bmatrix}}.\]
	We find that
	\begin{multline}
		F_0(a,b,c,k) =(-1)^{k+a+b+c}\ (a+2k)\\ \cdot\Gamma
		\left[ 
		\begin{matrix}
			a+k,\ b+k,\ c+k \\
			b,\ c,\ a-b+k+1,\ a-c+k+1,\ k+1,\ b+c-a
		\end{matrix}		
		\right].
		\label{4f3}
	\end{multline}
	Setting $u_a=u_b=u_c=u_k=1$, and $v_a=v_b=1/2,\ v_c=v_k=0$ in \eqref{4f3} to get a WZ pair:
	\[F(n,k)=\frac{(-1)^kn(6n+4k+1) (4n+2k)!^2}{2^{4n+2k}(2n+k)(2n+k)!(2n)!(n-1)!(n+k)!(2n+2k+1)!},\]
	\[G(n,k)=\frac{\alpha(n,k)(-1)^{k+1} (4n+2k)!^2}{2^{4n+2k+1}(2n+k)(2n+k)!(2n+1)!n!(n+k)!(2n+2k+1)!},\]
	where 
	\begin{multline*}
		\alpha(n,k)=168n^4+336n^3k+176n^3+ 240n^2k^2+264n^2k+58n^2\\+72nk^3+128nk^2+ 58nk+6n+8k^4+20k^3+14k^2+3k.
	\end{multline*}
	
	Clearly, $F(0,k) = 0$. For $n=p$, the classification of $k$ according to the computation of $v_p(F(p,k))$ are summarized in the following table:
	\begin{center}
		\begin{tabular}{|c|c|}
			\hline
			The range of $k$ & $p$-adic evaluation $v_p(F(p,k))$ \\ \hline
			$k=0$ & $1+8-1-2-2-1-2=1$ \\ \hline
			$0< k<\frac{p-1}{2}$ but $k\neq \frac{p-1}{4}$ & $1+8-2-2-1-2=2$ \\ \hline
			$k=\frac{p-1}{4}$ & $1+1+8-2-2-1-2=3$ \\ \hline
			$k=\frac{p-1}{2}$ & $1+8-2-2-1-3=1$ \\ \hline
			$\frac{p-1}{2}< k\leq p-1$ but $k\neq \frac{3p-1}{4}$ & $1+10-2-2-1-3=3$ \\ \hline
			$k=\frac{3p-1}{4}$ & $1+1+10-2-2-1-3=4$ \\ \hline
		\end{tabular}
	\end{center}
	In summary, $v_p(F(p,k))\geq 1$ for $0\leq k\leq p-1$. Thus
	\[F(p,k)\equiv F(0,k)\equiv 0\pmod{p}\quad \text{for}\quad 0\leq k\leq p-1.\]
	
	Now we evaluate the congruence $\sum_{n=0}^{p-1} G(n,k) \pmod{p}$ for $k=\frac{p-1}{2}$. We have
	\[G(n,\frac{p-1}{2})=\frac{\alpha(n,\frac{p-1}{2})(-1)^{\frac{p+1}{2}} (4n+p-1)!^2}{2^{4n+p}(2n+\frac{p-1}{2})(2n+\frac{p-1}{2})!(2n+1)!n!(n+\frac{p-1}{2})!(2n+p)!},\]
	where 
	\begin{multline*}
		\alpha(n,\frac{p-1}{2})=168n^4+(168p+8)n^3+(60p^2+12p-14)n^2\\+(9p^3+5p^2-8p)n+\frac{p^4}{2}+\frac{p^3}{2}-p^2. 
	\end{multline*}
	The case of $p=3$ can be verified directly. For $p>3$, the classification of $n$ according to the computation of $v_p(G(n,\frac{p-1}{2}))$  are summarized in the following table:
	\begin{center}
		\begin{tabular}{|c|c|}
			\hline
			The range of $n$ & $p$-adic evaluation $v_p(G(n,\frac{p-1}{2}))$ \\ \hline
			$0\leq n< \frac{p+1}{4}$ & $2-1=1$\\ \hline
			$n=\frac{p+1}{4}$ & $4-1-1-1=1$ \\ \hline
			$\frac{p+1}{4}< n< \frac{p-1}{2}$ & $4-1-1=2$ \\ \hline
			$n=\frac{p-1}{2}$ & $4-1-1-1=1$ \\ \hline
			$\frac{p+1}{2}\leq n< \frac{3p+1}{4}$ & $6-1-1-1-2=1$ \\ \hline
			$n=\frac{3p+1}{4}$ & $8-1-2-1-1-2=1$ \\ \hline
			$\frac{3p+1}{4}< n\leq p-1$ & $8-2-1-1-2=2$ \\ \hline
		\end{tabular}
	\end{center}
	In summary, $v_p(G(n,\frac{p-1}{2}))\geq 1$ for $0\leq n\leq p-1$. Thus
	\[G(n,\frac{p-1}{2})\equiv 0\pmod{p}\quad \text{for}\quad 0\leq n\leq p-1,\]
	which implies that
	\[\sum_{n=0}^{p-1}G(n,\frac{p-1}{2})\equiv 0\pmod{p}.\]
	Consequently, we obtain
	\begin{eqnarray*}
		\sum_{n=0}^{p-1}G(n,0)&=&\sum_{n=0}^{p-1}-\frac{42n^2+23n+3}{2^{4n+1}(2n+1)}\binom{4n}{2n}^2\binom{2n}{n}\\&\equiv& \sum_{n=0}^{p-1}\frac{42n^2+23n+3}{2^{4n}(2n+1)}\binom{4n}{2n}^2\binom{2n}{n}\equiv 0\pmod{p}.
	\end{eqnarray*}
	
	By the extend-Zeilberger's algorithm \cite{Chen09}, we find that the hypergeometric term
	\[t_n=\frac{252n^3 + 384n^2 + 139n + 15}{2^{4n}}\binom{4n}{2n}^2\binom{2n}{n}\]
	satisfies
	\[t_n-2\cdot G(n,0)=\Delta_n \left(\frac{n^3(4n)!^2}{2^{4n-2}(2n)!^3n!^2} \right).\]
	Summing over $n$ from $0$ to $p-1$, we derive that
	\[\sum_{n=0}^{p-1} \frac{252n^3 + 384n^2 + 139n + 15}{2^{4n}}\binom{4n}{2n}^2\binom{2n}{n} \equiv \sum_{n=0}^{p-1} G(n,0) \equiv 0 \pmod{p},\] 
	thereby proving Theorem \ref{42n^2+23n+3}.
\end{proof}	

\begin{proof}[Proof of Theorem \ref{binom{4n}{2n}binom{3n}{n}binom{2n}{n}}]
	Setting $u_a=u_k=-1,\ u_b=0,\ u_c=1$, and $v_a=v_b=v_k=1/2,\ v_c=0$ in \eqref{4f3} to get a WZ pair:
	\[F(n,k)=\frac{n(6n-4k-3)(2k)!(4n-2k-2)!(3n-k-2)!(2n-2k-2)!}{2^{6n-2k}(2n-k-1)!^2(n-k-1)!^2(2n)!(n-1)!k!},\]
	\[G(n,k)=-\frac{\alpha(n,k)(4n-2k)!(3n-k-1)!(2n-2k)!(2k)!}{2^{6n-2k+6}(2n-k+1)!(2n-k)!(n-k)!^2(2n+1)!n!k!},\]
	where 
	\begin{multline*}
		\alpha(n,k)=4k^4 - 44k^3n + 180k^2n^2 - 308kn^3 + 184n^4 - 12k^3 + 98k^2n\\ - 260kn^2  + 214n^3 + 11k^2 - 62kn + 79n^2 - 3k + 9n.
	\end{multline*}
	
	Clearly, $F(0,k) = 0$. For $n=p$, the classification of $k$ according to the computation of $v_p(F(p,k))$ are summarized in the following table:
	\begin{center}
		\begin{tabular}{|c|c|}
			\hline
			The range of $k$ & $p$-adic evaluation $v_p(F(p,k))$ \\ \hline
			$0\leq k<\frac{p}{2}-1$ but $k\neq \frac{p-3}{4}$ & $1+3+2+1-(2+2)=3$ \\ \hline
			$k=\frac{p-3}{4}$ & $1+1+3+2+1-(2+2)=4$ \\ \hline
			$\frac{p}{2}-1< k< p-1$ but $k\neq \frac{3p-3}{4}$ & $1+1+2+2-(2+2)=2$ \\ \hline
			$k=\frac{3p-3}{4}$ & $1+1+1+2+2-(2+2)=3$ \\ \hline
			$k=p-1$ & $1+1+2+1-(2+2)=1$ \\ \hline
		\end{tabular}
	\end{center}
	In summary, $v_p(F(p,k))\geq 1$ for $0\leq k\leq p-1$. Thus
	\[F(p,k)\equiv F(0,k)\equiv 0\pmod{p}\quad \text{for}\quad 0\leq k\leq p-1.\]
	
	Now we evaluate the congruence $\sum_{n=0}^{p-1} G(n,k) \pmod{p}$ for $k=\frac{p-1}{2}$. We have
	\[G(n,\frac{p+1}{2})=-\frac{\alpha(n,\frac{p+1}{2})(4n-p-1)!(3n-\frac{p+3}{2})!(2n-p-1)!(p+1)!}{2^{6n-p+5}(2n-\frac{p-1}{2})!(2n-\frac{p+1}{2})!(n-\frac{p+1}{2})!^2(2n+1)!n!(\frac{p+1}{2})!}.\]
	Notice that for $0\leq n< \frac{p+1}{2}$, $G(n,\frac{p-1}{2})=0$. For $\frac{p+1}{2}\leq n\leq p-1$, the cases of $p=3,\ 5,\ 7$ can be verified directly. For $p>7$, the classification of $n$ according to the computation of $v_p(G(n,\frac{p-1}{2}))$  are summarized in the following table:
	\begin{center}
		\begin{tabular}{|c|c|}
			\hline
			The range of $n$ & $p$-adic evaluation $v_p(G(n,\frac{p-1}{2}))$ \\ \hline
			$\frac{p+1}{2}\leq n<\frac{3p-1}{4}$ & $1+1+1-1=2$ \\ \hline
			$\frac{3p-1}{4}\leq n<\frac{3p+1}{4}$ & $1+1+1-(1+1)=1$ \\ \hline 
			$\frac{3p+1}{4}\leq n< \frac{5p+3}{6}$ & $2+1+1-(1+1+1)=1$ \\ \hline
			$\frac{5p+3}{6}\leq n\leq p-1$ & $2+2+1-(1+1+1)=2$ \\ \hline
		\end{tabular}
	\end{center}
	In summary, $v_p(G(n,\frac{p-1}{2}))\geq 1$ for $0\leq n\leq p-1$. Thus
	\[G(n,\frac{p-1}{2})\equiv 0\pmod{p}\quad \text{for}\quad 0\leq n\leq p-1,\]
	which implies that
	\[\sum_{n=0}^{p-1}G(n,\frac{p-1}{2})\equiv 0\pmod{p}.\]
	Consequently, we obtain
	\begin{eqnarray*}
		\sum_{n=0}^{p-1}G(n,0)&=&\sum_{n=0}^{p-1}-\frac{92n^2+61n+9}{2^{6n+6}\cdot 3\cdot (2n+1)}\binom{4n}{2n}\binom{3n}{n}\binom{2n}{n}\\&\equiv& \sum_{n=0}^{p-1}\frac{92n^2+61n+9}{2^{6n}(2n+1)}\binom{4n}{2n}\binom{3n}{n}\binom{2n}{n}\equiv 0\pmod{p},
	\end{eqnarray*}
	thereby proving Theorem \ref{binom{4n}{2n}binom{3n}{n}binom{2n}{n}}.
\end{proof}	

\begin{proof}[Proof of Theorem \ref{binom{4n}{2n}binom{3n}{n}binom{2n}{n}^3}]
	Setting $u_a=1,\ u_b=u_k=-1,\ u_c=0$, and $v_a=0,\ v_b=v_c=v_k=1/2$ in \eqref{3f21} to get a WZ pair:
	\[F(n,k)=\frac{(-1)^{n+k}n(2n+2k)!(2n-2k-2)!(4n)!(3n)!}{2^{12n}(n+k)!^2(2n-k-1)!(n-k-1)!^2n!^3(2n+k)!},\]
	\[G(n,k)=\frac{(-1)^{n+k}\alpha(n,k)(2n+2k)!(2n-2k)!(4n)!(3n)!}{2^{12n+8}(n+k)!^2(2n-k+1)!(n-k)!^2n!^3(2n+k+1)!},\]
	where
	\[
	\alpha(n,k)=16k^4 - 128k^2n^2 + 688n^4 - 76k^2n + 988n^3 - 16k^2 + 508n^2 + 111n + 9.
	\]
	
	Clearly, $F(0,k)=0$. For $n=p$, the classification of $k$ according to the computation of $v_p(F(p,k))$ are summarized in the following table:
	\begin{center}
		\begin{tabular}{|c|c|}
			\hline
			The range of $k$ & $p$-adic evaluation $v_p(F(p,k))$ \\ \hline
			$0\leq k< \frac{p-1}{2}$ & $1+2+1+4+3-(2+1+3+2)=3$ \\ \hline
			$\frac{p-1}{2}\leq k< \frac{p+1}{2}$ & $1+2+4+3-(2+1+3+2)=2$ \\ \hline
			$\frac{p+1}{2}\leq k\leq p-1$ & $1+3+4+3-(2+1+3+2)=3$ \\ \hline
		\end{tabular}
	\end{center}
	In summary, $v_p(F(p,k))\geq 2$ for $0\leq k\leq p-1$. Thus
	\[F(p,k)\equiv F(0,k)\equiv 0\pmod{p^2}\quad \text{for}\quad 0\leq k\leq p-1.\]
	
	Now we evaluate the congruence $\sum_{n=0}^{p-1} G(n,k) \pmod{p}$ for $k=\frac{p-1}{2}$. We have
	\[
	G(n,\frac{p-1}{2})=\frac{(-1)^{n+\frac{p-1}{2}}\alpha(n,\frac{p-1}{2})(2n+p-1)!(2n-p+1)!(4n)!(3n)!}{2^{12n+8}(n+\frac{p-1}{2})!^2(2n-\frac{p-1}{2}+1)!(n-\frac{p-1}{2})!^2n!^3(2n+\frac{p-1}{2}+1)!}.
	\]
	Notice that for $0\leq n<\frac{p-1}{2}$, $G(n,\frac{p-1}{2})=0$. For $\frac{p-1}{2}\leq n\leq p-1$, the cases of $p=3,\ 5,\ 7,\ 11$ can be verified directly. For $p\geq 13$, the classification of $n$ according to the computation of $v_p(G(n,\frac{p-1}{2}))$  are summarized in the following table:
	\begin{center}
		\begin{tabular}{|c|c|}
			\hline
			The range of $n$ & $p$-adic evaluation $v_p(G(n,\frac{p-1}{2}))$ \\ \hline
			$n=\frac{p-1}{2}$ & $1+1+1-1=2$ \\ \hline
			$\frac{p+1}{2}\leq n< \frac{2p+1}{3}$ & $2+2+1-(2+1)=2$ \\ \hline
			$\frac{2p+1}{3}\leq n< \frac{3p-3}{4}$ & $2+2+2-(2+1)=3$ \\ \hline
			$\frac{3p-3}{4}\leq n<\frac{3p-1}{4}$ & $2+2+2-(2+1+1)=2$ \\ \hline
			$\frac{3p-1}{4}\leq n<\frac{3p+1}{4}$ & $2+2+2-(2+1+2)=1$ \\ \hline
			$\frac{3p+1}{4}\leq n\leq p-1$ & $2+3+2-(2+1+2)=2$ \\ \hline
		\end{tabular}
	\end{center}
	In summary, $v_p(G(n,\frac{p-1}{2}))\geq 1$ for $0\leq n\leq p-1$. Thus
	\[G(n,\frac{p-1}{2})\equiv 0\pmod{p}\quad \text{for}\quad 0\leq n\leq p-1,\] which implies that
	\[\sum_{n=0}^{p-1}G(n,\frac{p-1}{2})\equiv 0\pmod{p}.\]
	Consequently, we obtain
	\begin{eqnarray*}
		\sum_{n=0}^{p-1}G(n,0)&=&\sum_{n=0}^{p-1}\frac{(-1)^n(172n^2+75n+9)}{2^{12n+8}}\binom{4n}{2n}\binom{3n}{n}\binom{2n}{n}^3 \\ &\equiv& \sum_{n=0}^{p-1}\frac{(-1)^n(172n^2+75n+9)}{2^{12n}}\binom{4n}{2n}\binom{3n}{n}\binom{2n}{n}^3 \equiv 0 \pmod{p},
	\end{eqnarray*}
	thereby proving Theorem \ref{binom{4n}{2n}binom{3n}{n}binom{2n}{n}^3}.
\end{proof}

\begin{proof}[Proof of Theorem \ref{binom{6n}{3n}binom{3n}{n}binom{2n}{n}}]
	Utilizing the $_2F_1$ hypergeometric identity in \cite{Ges95},
	\[_2F_1\left[\begin{matrix}\frac{a}{2},\  \frac{a+1}{2} \\b \end{matrix}\;;1\; \right]
	=2^{2b-a-2}\cdot \Gamma{\begin{bmatrix} b,\ b-a-\frac{1}{2}\\2b-a-1,\  \frac{1}{2}\end{bmatrix}}.\]
	We find that
	\begin{equation}
		F_0(a,b,k)=2^{-2k-a}\cdot \Gamma{\begin{bmatrix} a+2k,\ 2b-a-1,\ b-a\\a,\ b+k,\ k+1,\ 2b-2a-1 \end{bmatrix}}.\label{2f1_1}
	\end{equation}
	Setting $u_a=-1,\ u_b=1, u_k=1$, and $v_a=1/2,\ v_b=v_k=0$ in \eqref{2f1_1} to get a WZ pair:
	\[F(n,k)=\frac{(-1)^{n}n(2n+4k)!(6n-2)!}{2^{15n+6k}(n+2k)!(3n-1)!n!(2n+k-1)!(n+k)!},\]
	\[G(n,k)=\frac{(-1)^n\alpha(n,k)(2n+4k)!(6n-2)!}{2^{15n+6k+7}(2n-1)(n+2k)!(3n-1)!n!(2n+k)!(n+k)!},\]
	where \[\alpha(n,k)=616n^3+(528k-204)n^2+(128k^2-96k+6)n+12k+3.\]
	
	Clearly, $F(0,k)=0$. For $n=p$, the classification of $k$ according to the computation of $v_p(F(p,k))$ are summarized in the following table:
	\begin{center}
		\begin{tabular}{|c|c|}
			\hline
			The range of $k$ & $p$-adic evaluation $v_p(F(p,k))$ \\ \hline
			$k=0$ & $1+2+5-(1+2+1+1+1)=2$ \\ \hline
			$0<k<p/4$ & $1+2+5-(1+2+1+2+1)=1$ \\ \hline
			$p/4<k<p/2$ & $1+3+5-(1+2+1+2+1)=2$ \\ \hline
			$p/2<k<3p/4$ & $1+4+5-(2+2+1+2+1)=2$ \\ \hline
			$3p/4<k\leq p-1$ & $1+5+5-(2+2+1+2+1)=3$ \\ \hline
		\end{tabular}
	\end{center}
	In summary, $v_p(F(p,k))\geq 1$ for $0\leq k\leq p-1$. Thus
	\[F(p,k)\equiv F(0,k)\equiv 0\pmod{p}\quad \text{for}\quad 0\leq k\leq p-1.\]
	
	Now we evaluate the congruence $\sum_{n=0}^{p-1} G(n,k) \pmod{p}$ for $k=\frac{p-1}{2}$. We have
	\begin{eqnarray*}
		G(n,\frac{p-1}{2})&=&\frac{(-1)^n2^{-15n-3p-5}\alpha(n,\frac{p-1}{2})(2n+2p-2)!(6n-2)!}{2^{15n+3p+4}(2n-1)(n+p-1)!(3n-1)!n!(2n+\frac{p-1}{2})!(n+\frac{p-1}{2})!}\\
		&=&\frac{(-1)^n\alpha(n,\frac{p-1}{2})\binom{6n-2}{3n-1}\binom{3n-1}{n}\binom{2n+2p-2}{n+p-1}\binom{n+p-1}{\frac{p-1}{2}}}{2^{15n+3p+5}n(2n-1)\binom{2n+\frac{p-1}{2}}{\frac{p-1}{2}}},
	\end{eqnarray*}
	where \[\alpha(n,\frac{p-1}{2})=616n^3 + 264n^2p + 32np^2 - 468n^2 - 112np + 86n + 6p - 3.\]
	The case of $p=3$ can be verified directly. For $p>3$, the classification of $n$ according to the computation of $v_p(G(n,\frac{p-1}{2}))$  are summarized in the following table:
	\begin{center}
		\begin{longtable}{|c|c|}
			\hline
			The range of $n$ & $p$-adic evaluation $v_p(G(n,\frac{p-1}{2}))$ \\ \hline
			$n=0$ & $1$ \\ \hline
			$0< n< \frac{p+2}{6}$ & $2-1=1$\\ \hline   	
			$\frac{p+2}{6}< n< \frac{p+1}{4}$ & $2+1-1=2$ \\ \hline
			$\frac{p+1}{4}\leq n< \frac{p+1}{3}$ & $2+1-(1+1)=1$ \\ \hline
			$\frac{p+1}{3}\leq n< \frac{3p+2}{6}$ & $2+2-(1+1+1)=1$ \\ \hline
			$\frac{3p+2}{6}< n< \frac{p+1}{2}$ & $2+3-(1+1+1)=2$ \\ \hline
			$n=\frac{p+1}{2}$ & $1+2+3-(1+1+1+1+1)=1$ \\ \hline
			$\frac{p+1}{2}< n<\frac{p}{2}+1$ & $2+3-(1+1+1+1)=1$ \\ \hline
			$\frac{p}{2}+1< n<\frac{2p+1}{3}$ & $3+3-(1+1+1+1)=2$ \\ \hline
			$\frac{2p+1}{3}\leq n<\frac{3p+1}{4}$ & $3+4-(1+2+1+1)=2$ \\ \hline
			$\frac{3p+1}{4}\leq n<\frac{5p+2}{6}$ & $3+4-(1+2+2+1)=1$ \\ \hline
			$\frac{5p+2}{6}< n<p-1$ & $3+5-(1+2+2+1)=2$ \\ \hline
		\end{longtable}
	\end{center}
	In summary, $v_p(G(n,\frac{p-1}{2}))\geq 1$ for $0\leq n\leq p-1$. Thus
	\[G(n,\frac{p-1}{2})\equiv 0\pmod{p}\quad \text{for}\quad 0\leq n\leq p-1,\] which implies that
	\[\sum_{n=0}^{p-1}G(n,\frac{p-1}{2})\equiv 0\pmod{p}.\]
	Consequently, we obtain
	\[\sum_{n=0}^{p-1}G(n,0)=\sum_{n=0}^{p-1} \frac{(-1)^n(616n^3 - 204n^2 + 6n + 3)}{2^{15n+8}(6n-1)(2n-1)}\binom{6n}{3n}\binom{3n}{n}\binom{2n}{n} \equiv 0 \pmod{p}.\]
	By the extended-Zeilberger's algorithm \cite{Chen09}, we find that the hypergeometric term
	\[t_n = \frac{(-1)^n(154n+15)}{2^{15n}}\binom{6n}{3n}\binom{3n}{n}\binom{2n}{n}\]	satisfies
	\[-\frac{3}{4096}\cdot t_n + G(n,0) = \Delta_n \left(\frac{(-1)^{n+1}n^3(6n)!}{2^{15n}(2n-1)(6n-1)n!^3(3n)!} \right).\]
	Summing over $n$ from $0$ to $p-1$, we derive that
	\[\sum_{n=0}^{p-1} \frac{(-1)^n(154n+15)}{2^{15n}}\binom{6n}{3n}\binom{3n}{n}\binom{2n}{n}\equiv \sum_{n=0}^{p-1} G(n,0) \equiv 0 \pmod{p},\] 
	thereby proving Theorem \ref{binom{6n}{3n}binom{3n}{n}binom{2n}{n}}.
\end{proof}

\begin{proof}[Proof of Theorem \ref{binom{2n}{n}^5}]
	For \eqref{20n^2+8n+1} of Theorem \ref{binom{2n}{n}^5}, we utilize Dixon's theorem \cite{Bai35},
	\[_3F_2\left[\begin{matrix}a,\ b,\ c \\1+a-b,\ 1+a-c \end{matrix}\;;1\;\right]
	=\Gamma{\begin{bmatrix} 1+\frac{a}{2},1+a-b,1+a-c,1+\frac{a}{2}-b-c\\
			1+a,1+\frac{a}{2}-b,1+\frac{a}{2}-c,1+a-b-c\end{bmatrix}},\]
	we find that
	\begin{multline}
		F_0(a,b,c,d,k)  = (-1)^{a+b+c}\cdot \Gamma
		\left[ 
		\begin{matrix}
			2a+k,\ b+k,   \\
			2a-b+k+1,\ 2a-c+k+1,
		\end{matrix}
		\right. \\
		\left.
		\begin{matrix}
			c+k,\ b+c-a \\
			a,\ b,\ c,\ k+1,\ b-a,\ c-a,\ b+c-a
		\end{matrix}
		\right].
		\label{3f21}
	\end{multline}
	Setting $u_a=-1,\ u_b=u_c=0,\ u_k=1$, and $v_a=v_b=v_c=1/2,\ v_k=0$ in \eqref{3f21} to get a WZ pair:
	\[F(n,k)=\frac{(-1)^{3n+k}n^3(2n+2k)!^2(2n)!(2n-2k-2)!^2}{2^{12n}(n-k-1)!^3(n+k)!^3n!^4},\]
	\[G(n,k)=\frac{(-1)^{n+k}(20n^2+8n-4k^2+1)(2n+2k)!^2(2n)!(2n-2k)!^2}{2^{12n+9}(n-k)!^3(n+k)!^3n!^4}.\]
	
	Clearly, $F(0,k)=0$. For $n=p$, the classification of $k$ according to the computation of $v_p(F(p,k))$ are summarized in the following table:
	\begin{center}
		\begin{tabular}{|c|c|}
			\hline
			The range of $k$ & $p$-adic evaluation $v_p(F(p,k))$ \\ \hline
			$0\leq k< \frac{p-1}{2}$ & $3+4+2+2-(3+4)=4$ \\ \hline
			$\frac{p-1}{2}\leq k< \frac{p+1}{2}$ & $3+4+2-(3+4)=2$ \\ \hline
			$\frac{p+1}{2}\leq k\leq p-1$ & $3+6+2-(3+4)=4$ \\ \hline
		\end{tabular}
	\end{center}
	In summary, $v_p(F(p,k))\geq 2$ for $0\leq k\leq p-1$. Thus
	\[F(p,k)\equiv F(0,k)\equiv 0\pmod{p^2}\quad \text{for}\quad 0\leq k\leq p-1.\]
	
	Now we evaluate the congruence $\sum_{n=0}^{p-1} G(n,k) \pmod{p}$ for $k=\frac{p-1}{2}$. We have
	\begin{eqnarray*}
		G(n,\frac{p-1}{2})=\frac{(-1)^{n+\frac{p-1}{2}}(20n^2-p^2+8n+2p)(2n+p-1)!^2(2n)!(2n-p+1)!^2}{2^{12n+9}(n-\frac{p-1}{2})!^3(n+\frac{p-1}{2})!^3n!^4}.	  
	\end{eqnarray*}
	Notice that for $0\leq n<\frac{p-1}{2}$, $G(n,\frac{p-1}{2})=0$. For $\frac{p-1}{2}\leq n\leq p-1$, the classification of $n$ according to the computation of $v_p(G(n,\frac{p-1}{2}))$ are summarized in the following table:
	\begin{center}
		\begin{tabular}{|c|c|}
			\hline
			The range of $n$ & $p$-adic evaluation $v_p(G(n,\frac{p-1}{2}))$ \\ \hline
			$n=\frac{p-1}{2}$ & $2$ \\ \hline
			$\frac{p+1}{2}\leq n\leq p-1$ & $4+1-3=2$ \\ \hline
		\end{tabular}
	\end{center}
	In summary, $v_p(G(n,\frac{p-1}{2}))\geq 2$ for $0\leq n\leq p-1$. Thus
	\[G(n,\frac{p-1}{2})\equiv 0\pmod{p^2}\quad \text{for}\quad 0\leq n\leq p-1,\] which implies that
	\[\sum_{n=0}^{p-1}G(n,\frac{p-1}{2})\equiv 0\pmod{p^2}.\]
	Consequently, we obtain
	\begin{eqnarray*}
		\sum_{n=0}^{p-1}G(n,0)&=&\sum_{n=0}^{p-1}\frac{(-1)^n(20n^2+8n+1)}{2^{12n+9}}\binom{2n}{n}^5 \\ &\equiv& \sum_{n=0}^{p-1}\frac{(-1)^n(20n^2+8n+1)}{2^{12n}}\binom{2n}{n}^5 \equiv 0 \pmod{p^2},
	\end{eqnarray*}
	thereby proving \eqref{20n^2+8n+1} of Theorem \ref{binom{2n}{n}^5}.
	
	For \eqref{820n^2+180n+13} of Theorem \ref{binom{2n}{n}^5}, setting $u_a=u_b=u_c=u_d=u_k=-1$, and $v_a=v_b=v_c=v_d=v_k=1/2$ in \eqref{5f4} to get a WZ pair:
	\[F(n,k)=\frac{(-1)^{n}n(6n-4k-3)(2n)!^5(2n-2k-2)!^4}{2^{20n-8k}(2n-k-1)!^4n!^6(n-k-1)!^4},\]
	\[G(n,k)=\frac{(-1)^{n}\alpha(n,k)(2n)!^5(2n-2k)!^4}{2^{20n-8k+15}(2n-k+1)!^4n!^6(n-k)!^4},\]
	where
	\begin{multline*} 
		\alpha(n,k)=128k^6 - 1664k^5n + 9024k^4n^2 - 26112k^3n^3 + 42496k^2n^4 - 36736kn^5\\ + 13120n^6 - 576k^5 + 6336k^4n - 27904k^3n^2 + 61440k^2n^3 - 67264kn^4 + 29120n^5\\ + 1040k^4 - 9344k^3n + 31488k^2n^2 - 46784kn^3 + 25648n^4 - 960k^3 + 6656k^2n\\ - 15200kn^2 + 11296n^3 + 480k^2 - 2264kn + 2572n^2 - 124k + 284n + 13.
	\end{multline*}
	
	Clearly, $F(0,k)=0$. For $n=p$, the classification of $k$ according to the computation of $v_p(F(p,k))$ are summarized in the following table:
	\begin{center}
		\begin{tabular}{|c|c|}
			\hline
			The range of $k$ & $p$-adic evaluation $v_p(F(p,k))$ \\ \hline
			$0\leq k< \frac{p}{2}-1$ but $k\neq \frac{p-3}{4}$ & $1+10+4-(4+6)=5$ \\ \hline
			$k=\frac{p-3}{4}$ & $1+1+10+4-(4+6)=6$ \\ \hline
			$\frac{p}{2}-1< k\leq p-1$ but $k\neq \frac{3p-3}{4}$ & $1+10-(4+6)=1$ \\ \hline
			$k=\frac{p-3}{4}$ & $1+1+10-(4+6)=2$ \\ \hline
		\end{tabular}
	\end{center} 
	In summary, $v_p(F(p,k))\geq 1$ for $0\leq k\leq p-1$. Thus
	\[F(p,k)\equiv F(0,k)\equiv 0\pmod{p}\quad \text{for}\quad 0\leq k\leq p-1.\]
	
	Now we evaluate the congruence $\sum_{n=0}^{p-1} G(n,k) \pmod{p}$ for $k=\frac{p+1}{2}$. We have
	\begin{eqnarray*}
		G(n,\frac{p+1}{2})=\frac{(-1)^{n}\alpha(n,\frac{p+1}{2})(2n)!^5(2n-p-1)!^4}{2^{20n-4p+11}(2n-\frac{p+1}{2}+1)!^4n!^6(n-\frac{p+1}{2})!^4}. 
	\end{eqnarray*}
	Notice that for $0\leq n<\frac{p+1}{2}$, $G(n,\frac{p+1}{2})=0$. For $\frac{p+1}{2}\leq n\leq p-1$, the classification of $n$ according to the computation of $v_p(G(n,\frac{p+1}{2}))$ are summarized in the following table:
	\begin{center}
		\begin{tabular}{|c|c|}
			\hline
			The range of $n$ & $p$-adic evaluation $v_p(G(n,\frac{p+1}{2}))$ \\ \hline
			$\frac{p+1}{2}\leq n<\frac{3p-1}{4}$ & $5$ \\ \hline
			$\frac{3p-1}{4}\leq n\leq p-1$ & $5-4=1$ \\ \hline
		\end{tabular}
	\end{center}
	In summary, $v_p(G(n,\frac{p+1}{2}))\geq 1$ for $0\leq n\leq p-1$. Thus
	\[G(n,\frac{p+1}{2})\equiv 0\pmod{p}\quad \text{for}\quad 0\leq n\leq p-1,\] which implies that
	\[\sum_{n=0}^{p-1}G(n,\frac{p+1}{2})\equiv 0\pmod{p}.\]
	Consequently, we obtain
	\begin{eqnarray*}
		\sum_{n=0}^{p-1}G(n,0)&=&\sum_{n=0}^{p-1}\frac{(-1)^{n}(820n^2+180n+13)}{2^{20n+15}}\binom{2n}{n}^5 \\ &\equiv& \sum_{n=0}^{p-1}\frac{(-1)^n(820n^2+180n+13)}{2^{20n}}\binom{2n}{n}^5 \equiv 0 \pmod{p},
	\end{eqnarray*}
	thereby proving \eqref{820n^2+180n+13} of Theorem \ref{binom{2n}{n}^5}.
\end{proof}

\begin{proof}[Proof of Theorem \ref{binom{4n}{2n}^2binom{2n}{n}}]
	
	For \eqref{80n^3 + 96n^2 + 30n + 3} of Theorem \ref{binom{4n}{2n}^2binom{2n}{n}}, 
	setting $u_a=0,\ u_b=u_c=u_k=1$, and $v_a=v_b=v_c=0,\ v_k=1/2$ in \eqref{4f3} to get a WZ pair:
	\[F(n,k)=\frac{(-1)^{n+k}n(4n+2k)!^2k!^2}{2^{8n}(2n+k)!^2(n-1)!^2(2n)!(2k+1)!^2},\]
	\[G(n,k)=\frac{(-1)^{n+k}(40n^2+24nk+4k^2+24n+8k+3)(4n+2k)!^2k!^2}{2^{8n+6}(2n+k)!^2n!^2(2n+1)!(2k)!^2}.\]
	
	Clearly, $F(0,k) = 0$. For $n=p$, the classification of $k$ according to the computation of $v_p(F(p,k))$ are summarized in the following table:
	\begin{center}
		\begin{tabular}{|c|c|}
			\hline
			The range of $k$ & $p$-adic evaluation $v_p(F(p,k))$ \\ \hline
			$0\leq k<\frac{p-1}{2}$ & $1+8-4-2=3$ \\ \hline
			$k=\frac{p-1}{2}$ & $1+8-4-2-2=1$ \\ \hline
			$\frac{p-1}{2}< k\leq p-1$ & $1+10-4-2-2=3$ \\ \hline
		\end{tabular}
	\end{center}
	In summary, $v_p(F(p,k))\geq 1$ for $0\leq k\leq p-1$. Thus
	\[F(p,k)\equiv F(0,k)\equiv 0\pmod{p}\quad \text{for}\quad 0\leq k\leq p-1.\]
	
	Now we evaluate the congruence $\sum_{n=0}^{p-1} G(n,k) \pmod{p}$ for $k=\frac{p-1}{2}$. We have
	\[G(n,\frac{p-1}{2})=\frac{(-1)^{n+\frac{p-1}{2}}(40n^2+12np+12n+p^2+2p)(4n+p-1)!^2(\frac{p-1}{2})!^2}{2^{8n+6}(2n+\frac{p-1}{2})!^2n!^2(2n+1)!(p-1)!^2}.\]
	The case of $p=3$ can be verified directly. For $p>3$, the classification of $n$ according to the computation of $v_p(G(n,\frac{p-1}{2}))$ are summarized in the following table:
	\begin{center}
		\begin{tabular}{|c|c|}
			\hline
			The range of $n$ & $p$-adic evaluation $v_p(G(n,\frac{p-1}{2}))$ \\ \hline
			$n=0$ & $1$\\ \hline
			$0< n< \frac{p+1}{4}$ & $2$ \\ \hline
			$\frac{p+1}{4}\leq n< \frac{p-1}{2}$ & $4-2=2$ \\ \hline
			$n=\frac{p-1}{2}$ & $4-2-1=1$ \\ \hline
			$\frac{p+1}{2}\leq n< \frac{3p+1}{4}$ & $6-2-1=3$ \\ \hline
			$\frac{3p+1}{4}\leq n\leq p-1$ & $8-4-1=3$ \\ \hline
		\end{tabular}
	\end{center}
	In summary, $v_p(G(n,\frac{p-1}{2}))\geq 1$ for $0\leq n\leq p-1$. Thus
	\[G(n,\frac{p-1}{2})\equiv 0\pmod{p}\quad \text{for}\quad 0\leq n\leq p-1,\] which implies that
	\[\sum_{n=0}^{p-1}G(n,\frac{p-1}{2})\equiv 0\pmod{p}.\]
	Consequently, we obtain
	\begin{eqnarray*}
		\sum_{n=0}^{p-1}G(n,0)&=&\sum_{n=0}^{p-1}\frac{(-1)^{n} (40n^2+24n+3)}{2^{8n+6}(2n+1)}\binom{4n}{2n}^2\binom{2n}{n}\\&\equiv& \sum_{n=0}^{p-1}\frac{(-1)^{n} (40n^2+24n+3)}{2^{8n}(2n+1)}\binom{4n}{2n}^2\binom{2n}{n}\equiv 0\pmod{p}.
	\end{eqnarray*}
	By the extended-Zeilberger's algorithm \cite{Chen09}, we find that the hypergeometric term
	\[t_n = \frac{(-1)^{n}(80n^3 + 96n^2 + 30n + 3)}{2^{8n}}\binom{4n}{2n}^2\binom{2n}{n}\]	satisfies
	\[2^{-5}\cdot t_n + G(n,0) = \Delta_n \left(\frac{(-1)^{n+1}n^3(4n)!^2}{2^{8n+1}(2n)!^3n!^2} \right).\]
	Summing over $n$ from $0$ to $p-1$, we derive that
	\[\sum_{n=0}^{p-1} \frac{(-1)^{n}(80n^3 + 96n^2 + 30n + 3)}{2^{8n}}\binom{4n}{2n}^2\binom{2n}{n}\equiv \sum_{n=0}^{p-1} G(n,0) \equiv 0 \pmod{p},\] 
	thereby proving \eqref{80n^3 + 96n^2 + 30n + 3} of Theorem \ref{binom{4n}{2n}^2binom{2n}{n}}.
	
	For \eqref{-96n^3 + 48n^2 + 26n + 3} of Theorem \ref{binom{4n}{2n}^2binom{2n}{n}}, utilizing Gauss's second summation theorem \cite{Bai35},
	\[_2F_1\left[\begin{matrix}a,\ b \\\frac{a+b+1}{2} \end{matrix}\;;\frac{1}{2}\; \right]
	=\Gamma{\begin{bmatrix} \frac{1}{2},\ \frac{a+b+1}{2}\\\frac{a+1}{2},\  \frac{b+1}{2}\end{bmatrix}}.\]
	We find that
	\begin{equation}
		F_0(a,b,k) =2^k\ \Gamma
		\left[ 
		\begin{matrix}
			2a+k,\ 2b+k,\ a+b+k+1 \\
			a,\ b,\ 2a+2b+2k+1,\ k+1
		\end{matrix}		
		\right].
		\label{2f1_2}
	\end{equation}
	Setting $u_a=u_b=1,\ u_k=0$, and $v_a=v_b=0,\ v_k=1/2$ in \eqref{2f1_2} to get a WZ pair:
	\[F(n,k)=\frac{n^2(4n+2k)!^2k!}{2^{12n+3k}(2n+k)!^3n!^2(2k+1)!},\]
	\[G(n,k)=-\frac{(48n^2+32nk+4k^2+32n+8k+3)(4n+2k)!^2k!}{2^{12n+3k+6}(2n+k+1)!(2n+k)!^2n!^2(2k)!}.\]
	
	Clearly, $F(0,k) = 0$. For $n=p$, the classification of $k$ according to the computation of $v_p(F(p,k))$ are summarized in the following table:
	\begin{center}
		\begin{tabular}{|c|c|}
			\hline
			The range of $k$ & $p$-adic evaluation $v_p(F(p,k))$ \\ \hline
			$0\leq k<\frac{p-1}{2}$ & $2+8-6-2=2$ \\ \hline
			$k=\frac{p-1}{2}$ & $2+8-1-6-2=1$ \\ \hline
			$\frac{p-1}{2}< k\leq p-1$ & $2+10-6-1-2=3$ \\ \hline
		\end{tabular}
	\end{center}
	In summary, $v_p(F(p,k))\geq 1$ for $0\leq k\leq p-1$. Thus
	\[F(p,k)\equiv F(0,k)\equiv 0\pmod{p}\quad \text{for}\quad 0\leq k\leq p-1.\]
	
	Now we evaluate the congruence $\sum_{n=0}^{p-1} G(n,k) \pmod{p}$ for $k=\frac{p-1}{2}$. We have
	\[G(n,\frac{p-1}{2})=-\frac{(48n^2+16np+16n+p^2+2p)(4n+p-1)!^2(\frac{p-1}{2})!}{2^{12n+3p/2+2}(2n+\frac{p+1}{2})!(2n+\frac{p-1}{2})!^2n!^2(p-1)!}.\]
	The case of $p=3$ can be verified directly. For $p>3$, the classification of $n$ according to the computation of $v_p(G(n,\frac{p-1}{2}))$ are summarized in the following table:
	\begin{center}
		\begin{tabular}{|c|c|}
			\hline
			The range of $n$ & $p$-adic evaluation $v_p(G(n,\frac{p-1}{2}))$ \\ \hline
			$n=0$ & $1$ \\ \hline
			$0< n< \frac{p-1}{4}$ & $2$\\ \hline
			$\frac{p-1}{4}\leq n<\frac{p+1}{4}$ & $2-1=1$ \\ \hline
			$\frac{p+1}{4}\leq n< \frac{p+1}{2}$ & $4-1-2=1$ \\ \hline
			$\frac{p+1}{2}\leq n< \frac{3p-1}{4}$ & $6-1-2=3$ \\ \hline
			$\frac{3p-1}{4}\leq n<\frac{3p+1}{4}$ & $6-2-2=2$ \\ \hline
			$\frac{3p+1}{4}\leq n\leq p-1$ & $8-2-4=2$ \\ \hline
		\end{tabular}
	\end{center}		
	In summary, $v_p(G(n,\frac{p-1}{2}))\geq 1$ for $0\leq n\leq p-1$. Thus
	\[G(n,\frac{p-1}{2})\equiv 0\pmod{p}\quad \text{for}\quad 0\leq n\leq p-1,\] which implies that
	\[\sum_{n=0}^{p-1}G(n,\frac{p-1}{2})\equiv 0\pmod{p}.\]
	Consequently, we obtain
	\begin{eqnarray*}
		\sum_{n=0}^{p-1}G(n,0)&=&\sum_{n=0}^{p-1}-\frac{48n^2+32n+3}{2^{12n+6}(2n+1)}\binom{4n}{2n}^2\binom{2n}{n}\\&\equiv& \sum_{n=0}^{p-1}\frac{48n^2+32n+3}{2^{12n}(2n+1)}\binom{4n}{2n}^2\binom{2n}{n}\equiv 0\pmod{p}.
	\end{eqnarray*}
	By the extended-Zeilberger's algorithm \cite{Chen09}, we find that the hypergeometric term
	\[t_n = \frac{-96n^3 + 48n^2 + 26n + 3}{2^{12n}}\binom{4n}{2n}^2\binom{2n}{n}\]	satisfies
	\[\frac{1}{16}\cdot t_n + G(n,0) = \Delta_n \left(\frac{n^3(4n)!^2}{2^{12n-3}(2n)!^3n!^2} \right).\]
	Summing over $n$ from $0$ to $p-1$, we derive that
	\[\sum_{n=0}^{p-1} \frac{-96n^3 + 48n^2 + 26n + 3}{2^{12n}}\binom{4n}{2n}^2\binom{2n}{n}\equiv \sum_{n=0}^{p-1} G(n,0) \equiv 0 \pmod{p},\]
	thereby proving \eqref{-96n^3 + 48n^2 + 26n + 3} of Theorem \ref{binom{4n}{2n}^2binom{2n}{n}}.
\end{proof}

\end{document}